\documentclass[12pt,a4paper]{article}
\usepackage{amssymb}
\usepackage{amsmath}
\usepackage{amsthm}
\usepackage{indentfirst}
\usepackage{ot1patch}

\theoremstyle{plain}
\newtheorem{tw}{Theorem}[section]

\newtheorem{lm}[tw]{Lemma}
\newtheorem{cor}[tw]{Corollary}

\theoremstyle{remark}
\newtheorem{df}[tw]{Definition}

\newtheorem{rem}[tw]{Remark}

\DeclareMathOperator{\cha}{char}

\DeclareMathOperator{\jac}{jac}
\DeclareMathOperator{\dgcd}{dgcd}
\DeclareMathOperator{\Irr}{Irr}
\DeclareMathOperator{\Sqf}{Sqf}
\DeclareMathOperator{\Img}{Im}

\author{Piotr J\k{e}drzejewicz, Janusz Zieli\'nski}
\title{Analogs of Jacobian conditions for subrings}
\date{}

\begin{document}

\maketitle

\begin{abstract}
We present a generalization of the Jacobian Conjecture
for $m$ polynomials in $n$ variables:
$f_1,\dots,f_m\in k[x_1,\dots,x_n]$,
where $k$ is a field of characteristic zero
and $m\in\{1,\dots,n\}$.
We express the generalized Jacobian condition
in terms of irreducible and square-free elements
of the subalgebra $k[f_1,\dots,f_m]$.
We also discuss obtained properties in a~more general setting
-- for subrings of unique factorization domains.
\end{abstract}

\begin{table}[b]\footnotesize\hrule\vspace{1mm}
\end{table}

\section*{Introduction}

The Jacobian Conjecture asserts that
if $k$ is a field of characteristic zero
and polynomials $f_1,\dots,f_n\in k[x_1,\dots,x_n]$
satisfy the Jacobian condition
$$\jac(f_1,\dots,f_n)\in k\setminus\{0\}\leqno (1)$$
(where $\jac$ denotes the Jacobian determinant),
then $k[f_1,\dots,f_n]=k[x_1,$ $\dots,$ $x_n]$.
In terms of endomorphisms of the polynomial algebra
$k[x_1,\dots,x_n]$: if a~$k$-endomorphism $\varphi$
satisfies the Jacobian condition
$$\jac(\varphi(x_1),\dots,\varphi(x_n))\in k\setminus\{0\},
\leqno (2)$$ then $\varphi$ is an automorphism.
For more information on the Jacobian Conjecture
we refer the reader to van den Essen's book \cite{EssenBook}.

\bigskip

In Section \ref{conjecture} we present and discuss
the following generalization of the Jacobian Conjecture,
denoted by $\text{JC}(m,n,k)$,
where $k$ is a field of characteristic zero,
$n$ is a positive integer, $m\in\{1,\dots,n\}$
and $\jac$ denotes the Jacobian determinant with respect
to given variables:
\begin{quote}
{\em "For arbitrary polynomials
$f_1,\dots,f_m\in k[x_1,\dots,x_n]$, if}
\end{quote}
$$\gcd\big(\jac^{f_1,\dots,f_m}_{x_{i_1},\dots,x_{i_m}},\;
1\leqslant i_1<\ldots<i_m\leqslant n\big)\in k\setminus\{0\},
\leqno (3)$$
\begin{quote}
{\em then $k[f_1,\dots,f_m]$ is a ring of constants
of some $k$-derivation of $k[x_1,\dots,x_n]$."}
\end{quote}

\medskip

This conjecture can be expressed in terms of polynomial
homomorphisms (and algebraic closedness) in the following way:
\begin{quote}
{\em "For every $k$-homomorphism
$\varphi\colon k[x_1,\dots,x_n]\rightarrow k[x_1,\dots,x_m]$, if
$$\gcd\big(\jac^{\varphi(x_1),\dots,\varphi(x_m)}_{x_{i_1},
\dots,x_{i_m}},\;1\leqslant i_1<\ldots<i_m\leqslant n\big)\in
k\setminus\{0\},$$
then $\Img\varphi$ is algebraically closed in $k[x_1,\dots,x_n]$."}
\end{quote}

\medskip

One of the authors obtained in \cite{keller} a characterization
of endomorphisms satisfying the Jacobian condition $(2)$,
where $k$ is a field of characteristic zero,
as mapping irreducible polynomials to square-free ones.
De Bondt and Yan proved in \cite{BondtYan} that
mapping square-free polynomials to square-free ones
is also equivalent to $(2)$.
We can express it in terms of polynomials
$f_1=\varphi(x_1)$, $\dots$, $f_n=\varphi(x_n)$:
condition $(1)$ holds if and only if all irreducible
(resp.\ all square-free) elements of the ring $k[f_1,\dots,f_n]$
are square-free in the ring $k[x_1,\dots,x_n]$.
In Theorem \ref{t2} we generalize this fact
for $m$ polynomials $f_1,\dots,f_m\in k[x_1,\dots,x_n]$,
where $m\in\{1,\dots,n\}$.
Namely, the generalized Jacobian condition $(3)$
is equivalent to each of the following ones:
$$\begin{array}{l}
\Irr k[f_1,\dots,f_m]\subset\Sqf k[x_1,\dots,x_n],\\[1ex]
\Sqf k[f_1,\dots,f_m]\subset\Sqf k[x_1,\dots,x_n],
\end{array}\leqno (4)$$
where $\Irr$ and $\Sqf$ denote the sets of irreducible
and square-free elements of the respective ring.
This fact is a consequence of a multidimensional
generalization of Freudenburg's lemma (\cite{FreudenburgNote})
obtained in Theorem~\ref{t1}.
A~presentation of succeding generalizations of this lemma
can be found in the Intoduction to \cite{charpbases}.

\bigskip

The above conjecture motivates us in Section \ref{analogs}
to consider the following properties for a~subring $R$
of a~unique factorization domain $A$:
$$\Irr R\subset\Sqf A,\quad \Sqf R\subset\Sqf A.\leqno (5)$$
In Theorem \ref{t3}, under some additional assumptions,
we express the second condition in a kind of factoriality:
$$\text{{\em "For every}}\;x\in A,\;y\in\Sqf A,\;
\text{{\em if}}\;x^2y\in R\setminus\{0\},\;
\text{{\em then}}\;x,y\in R."\leqno (6)$$
We call a subring $R$ satisfying condition $(6)$
square-factorially closed in~$A$.
In Theorem \ref{t4} we show that, under the same assumptions,
square-factorially closed subrings are root closed.

\section{\boldmath A generalization of the Jacobian Conjecture
for $m$ polynomials in $n$ variables}
\label{conjecture}

Let $k$ be a field of characteristic zero.
By $k[x_1,\dots,x_n]$ we denote the $k$-algebra
of polynomials in $n$ variables.

\medskip

Recall from \cite{jaccond} the following notion of
a "differential gcd" for $m$ polynomials
$f_1,\dots,f_m\in k[x_1,\dots,x_n]$, $m\in\{1,\dots,n\}$:
$$\dgcd(f_1,\dots,f_m)=
\gcd\big(\jac^{f_1,\dots,f_m}_{x_{i_1},\dots,x_{i_m}},\;
1\leqslant i_1<\ldots<i_m\leqslant n\big),$$
where $\jac^{f_1,\dots,f_m}_{x_{i_1},\dots,x_{i_m}}$
denotes the Jacobian determinant of $f_1$, $\dots$, $f_m$
with respect to $x_{i_1}$, $\dots$, $x_{i_m}$.
For $m=n$ we have
$$\dgcd(f_1,\dots,f_n)\sim\jac(f_1,\dots,f_n),$$
for $m=1$ we have
$$\dgcd(f)\sim\gcd\left(\dfrac{\partial f}{\partial x_1},\dots,
\dfrac{\partial f}{\partial x_n}\right),$$
where $g\sim h$ means that polynomials $g$ and $h$
are associated.
We put $\dgcd(f_1,\dots,f_m)=0$
if $\jac^{f_1,\dots,f_m}_{x_{i_1},\dots,x_{i_m}}=0$
for every $i_1$, $\dots$, $i_m$, that is,
$f_1$, $\dots$, $f_m$ are algebraically dependent over $k$.

\medskip

Let $k$ be a field and let $A$ be a $k$-algebra.
A $k$-linear map $d\colon A\to A$ such that
$d(fg)=d(f)g+fd(g)$ for $f,g\in A$,
is called a {\em $k$-derivation} of $A$.
The kernel of $d$ is denoted by $A^d$
and called the {\em ring of constants} of $d$.
For more information on derivations and their rings
of constants we refer the reader to Nowicki's book
\cite{Npolder}.

\medskip

Consider the following conjecture for $m$ polynomials
in $n$ variables.

\medskip

\noindent
{\bf\boldmath Conjecture $\text{JC}(m,n,k)$.}
{\em For arbitrary polynomials
$f_1,\dots,f_m\in k[x_1,$ $\dots,$ $x_n]$,
where $k$ is a field of characteristic zero
and $m\in\{1,\dots,n\}$,
if $$\dgcd(f_1,\dots,f_m)\in k\setminus\{0\},$$
then $k[f_1,\dots,f_m]$ is a ring of constants
of some $k$-derivation of $k[x_1,\dots,x_n]$.}

\medskip

Recall Nowicki's characterization of rings of constants
(\cite{Nrings}, Theorem~5.4,
\cite{Npolder}, Theorem~4.1.4, p.\ 47).

\medskip

\noindent
{\bf Theorem (Nowicki, 1994).}
{\em Let $A$ be a finitely generated $k$-domain,
where $k$ is a field of characteristic zero.
Let $R$ be a $k$-subalgebra of~$A$.
The following conditions are equivalent:

\medskip

\noindent
{\rm (i)} \
$R$ is a ring of constants of some $k$-derivation of $A$,

\medskip

\noindent
{\rm (ii)} \
$R$ is integrally closed in $A$ and $R_0\cap A=R$.}

\medskip

Let $D$ be a family of $k$-derivations
of a finitely generated $k$-domain $A$,
where $k$ is a field of characteristic zero.
It follows from Nowicki's Theorem that
the ring $$A^D=\bigcap_{d\in D}A^d$$
is a ring of constants of some single $k$-derivation of $A$
(\cite{Nrings}, Theorem~5.5,
\cite{Npolder}, Theorem~4.1.5, p.\ 47).

\medskip

Daigle observed (\cite{DaigleBook}, 1.4) that
condition {\rm (ii)} of Nowicki's Theorem
can be shortened to the following form:

\medskip

\noindent
{\rm (iii)} \
{\em $R$ is algebraically closed in $A$.}

\medskip

Now we see for example that conjecture $\text{JC}(2,3,k)$ 
asserts that if polynomials $f,g\in k[x,y,z]$ 
satisfy the condition
$$\gcd\big(\jac^{f,\,g}_{x,\,y},\jac^{f,\,g}_{x,\,z},
\jac^{f,\,g}_{y,\,z}\big)\in k\setminus\{0\},$$
then $k[f,g]$ is algebraically closed in $k[x,y,z]$.

\medskip

Let us note some basic observations according to
conjecture $\text{JC}(m,n,k)$.

\medskip

\begin{lm}
\label{l2}
$\text{\rm JC}(m,n,k)$ implies the Jacobian Conjecture
for $m$ variables over $k$.
\end{lm}

\begin{proof}
Assume that $\text{JC}(m,n,k)$ holds and consider
polynomials $f_1,\dots,f_m\in k[x_1,\dots,x_m]$ such that
$\jac^{f_1,\dots,f_m}_{x_1,\dots,x_m}\in k\setminus\{0\}$.

\medskip

In $k[x_1,\dots,x_n]$ we have $\dgcd(f_1,\dots,f_m)=
\jac^{f_1,\dots,f_m}_{x_1,\dots,x_m}$,
so $k[f_1,\dots,f_m]$ is algebraically closed
in $k[x_1,\dots,x_n]$ by $\text{JC}(m,n,k)$.
Hence, $k[f_1,\dots,f_m]$ is algebraically closed
in $k[x_1,\dots,x_m]$.
And then $k[f_1,\dots,f_m]=k[x_1,\dots,x_m]$, because
$f_1,\dots,f_m$ are algebraically independent over $k$.
\end{proof}

Now, recall from \cite{Njacobian} and \cite{NNagata} that
a polynomial $f\in k[x_1,\dots,x_n]$ over a field $k$
is called {\em closed} if the ring $k[f]$
is integrally closed in $k[x_1,\dots,x_n]$.
When $\cha k=0$, a polynomial $f$ is closed
if and only if $k[f]$ is a ring of constants
of some $k$-derivation of $k[x_1,\dots,x_n]$
(\cite{Njacobian}, Theorem 2.1,
\cite{Npolder}, Theorem~7.1.4, p.\ 80).
Necessary and sufficient conditions for a polynomial
to be closed were collected and completed by Arzantsev
and Petravchuk in \cite{ArzhantsevPetravchuk}, Theorem 1.
Ayad proved (\cite{Ayad}, Proposition 14) that a polynomial
$f\in k[x,y]$, where $\cha k=0$, is closed if
$$\gcd\left(\dfrac{\partial f}{\partial x},
\dfrac{\partial f}{\partial y}\right)
\in k\setminus\{0\}.$$
His proof can be generalized to $n$ variables
(it was noted in \cite{closed}, Proposition 4.2),
so we obtain the following fact.

\medskip

\noindent
{\bf Theorem (Ayad, 2002).}
$\text{\rm JC}(1,n,k)$ {\em is true.}

\medskip

\begin{rem}
\label{r1}
The reverse implication in $\text{JC}(m,n,k)$
need not to be true if $m<n$.
\end{rem}

In this case $n\geqslant 2$.
As an example we may consider a polynomial
$f_1=x_1^2x_2\in k[x_1,\dots,x_n]$ and, if $m>1$,
polynomials $f_2=x_3$, $\dots$, $f_m=x_{m+1}$.

\medskip

We have
$$\dgcd(x_1^2x_2,x_3,\dots,x_{m+1})=\gcd(2x_1x_2,x_1^2)=x_1,$$
so $\dgcd(x_1^2x_2,x_3,\dots,x_{m+1})\not\in k\setminus\{0\}$.

\medskip

On the other hand, $k[x_1^2x_2,x_3,\dots,x_{m+1}]$
is algebraically closed in $k[x_1,$ $\dots,$ $x_n]$
as a ring of constants of a family of derivations
$$\left\{x_1\dfrac{\partial}{\partial x_1}
-2x_2\dfrac{\partial}{\partial x_2},
\dfrac{\partial}{\partial x_{m+2}},\dots,
\dfrac{\partial}{\partial x_n}\right\}.$$

\section{Analogs of Jacobian conditions in terms of
irreducible and square-free elements}
\label{analogs}

If $R$ is a commutative ring with unity, then $R^{\ast}$
denotes the set of non-invertible elements of $R$.
An element $a\in R$ is called {\em square-free} if
it cannot be presented in the form $a=b^2c$,
where $b,c\in R$ and $b\not\in R^{\ast}$.
By $\Irr R$ we denote the set of irreducible elements
of $R$ and by $\Sqf R$ we denote the set of square-free
elements of $R$.

\medskip

Let $k$ be a field of characteristic zero.
Consider arbitrary polynomials $f_1,\dots,f_m\in
k[x_1,\dots,x_n]$, where $m\in\{1,\dots,n\}$.
Let $g\in k[x_1,\dots,x_n]$ be an irreducible polynomial.

\medskip

The following lemma is a natural generalization of
\cite{keller}, Lemma 3.1.
For the proof it is enough to add the argument from
the beginning of the proof of
\cite{charpbases}, Proposition 3.4.a with $Q=(g)$.

\begin{lm}
\label{l1}
For a given $i\in\{1,\dots,m\}$ consider the following condition:
$$\begin{array}{l}
\mbox{there exist $s_1,\dots,s_m\in k[x_1,\dots,x_n]$,
where $g\nmid s_i$, such that}\\
\mbox{$g\mid s_1d(f_1)+\ldots+s_md(f_m)$
for every $k$-derivation $d$ of $k[x_1,\dots,x_n]$.}
\end{array}\leqno (\ast)$$

\medskip

\noindent
{\bf a)}
Then $g\mid\dgcd(f_1,\dots,f_m)$ if and only if
condition $(\ast)$ holds for some $i\in\{1,\dots,m\}$.

\medskip

\noindent
{\bf b)}
If, for a given $i\in\{1,\dots,m\}$,
condition $(\ast)$ holds,
then $\overline{f_i}$ is algebraic over the field
$k(\,\overline{f_1},\dots,\overline{f_{i-1}},
\overline{f_{i+1}},\dots,\overline{f_m}\,)$.
\end{lm}

Note the following consequence of Lemma \ref{l1}
and \cite{keller}, Lemma 3.2.b
(where the polynomial $w$ can be obtained as irreducible).

\begin{cor}
\label{c1}
If $g\mid\dgcd(f_1,\dots,f_m)$, then $g\mid w(f_1,\dots,f_m)$
for some irreducible polynomial $w\in k[x_1,\dots,x_m]$.
\end{cor}

The following theorem is a multidimensional generalization
of Freudenburg's lemma (\cite{FreudenburgNote}).

\begin{tw}
\label{t1}
Let $k$ be a field of characteristic zero,
let $f_1,\dots,f_m\in k[x_1,$ $\dots,x_n]$ be arbitrary
polynomials, where $m\in\{1,\dots,n\}$,
and let $g\in k[x_1,$ $\dots,$ $x_n]$ be an irreducible polynomial.
The following conditions are equivalent:

\medskip

\noindent
{\rm (i)} \
$g\mid\dgcd(f_1,\dots,f_m)$,

\medskip

\noindent
{\rm (ii)} \
$g^2\mid w(f_1,\dots,f_m)$ for some irreducible polynomial
$w\in k[x_1,\dots,x_m]$,

\medskip

\noindent
{\rm (iii)} \
$g^2\mid w(f_1,\dots,f_m)$ for some square-free polynomial
$w\in k[x_1,\dots,x_m]$.
\end{tw}

\begin{proof}
Implication $\text{(ii)}\Rightarrow \text{(i)}$
was already proved in the proof of \cite{keller},
Theorem 4.1, $\text{(ii)}\Rightarrow \text{(i)}$.
Implication $\text{(ii)}\Rightarrow \text{(iii)}$ is trivial.

\medskip

\noindent
$\text{(i)}\Rightarrow \text{(ii)}$ \
We combine the arguments from proofs of \cite{keller},
Theorem 4.1, $\text{(i)}\Rightarrow \text{(ii)}$
and \cite{charpbases}, Theorem 3.6 $(\Rightarrow)$.
Assume that $g\mid\dgcd(f_1,\dots,f_m)$.
By Corollary \ref{c1}, $g\mid w(f_1,\dots,f_m)$ 
for some irreducible polynomial $w\in k[x_1,\dots,x_m]$.
We proceed like in \cite{keller}, using a derivation 
$d(f)=\jac^{f_1,\dots,f_{m-1},f}_{x_{j_1},\ldots,x_{j_m}}$
for arbitrary $j_1<\ldots<j_m$, instead of $d_n$, 
and applying Lemma \ref{l1} instead of 
\cite{keller}, Lemma 3.1. 

\medskip

\noindent
$\text{(iii)}\Rightarrow \text{(ii)}$ \
We apply the proof of \cite{BondtYan}, Theorem 2.1, 
$3)\Rightarrow 2)$. 
Assume that $g^2\mid w(f_1,\dots,f_m)$ 
for some square-free polynomial $w\in k[x_1,\dots,x_m]$,
so $w(f_1,\dots,f_m)=g^2h$, where $h\in k[x_1,\dots,x_n]$.
Then there exist polynomials $w_1,w_2\in k[x_1,\dots,x_m]$
such that $w=w_1w_2$, $w_1$ is irreducible and 
$g\mid w_1(f_1,\dots,f_m)$.
Then we proceed like in \cite{BondtYan}.
\end{proof}

As a consequence of Theorem \ref{t1} we obtain the following
generalization of \cite{keller}, Theorem 5.1 
and \cite{BondtYan}, Corollary 2.2.

\begin{tw}
\label{t2}
Let $A=k[x_1,\dots,x_n]$, where $k$ is a field
of characteristic zero.
Assume that $f_1,\dots,f_m\in A$ are algebraically
independent over $k$, where $m\in\{1,\dots,n\}$.
Put $R=k[f_1,\dots,f_m]$.
Then the following conditions are equivalent:

\medskip

\noindent
{\rm (i)} \
$\dgcd(f_1,\dots,f_m)\in k\setminus\{0\}$,

\medskip

\noindent
{\rm (ii)} \
$\Irr R\subset\Sqf A$,

\medskip

\noindent
{\rm (iii)} \
$\Sqf R\subset\Sqf A$.
\end{tw}

The above theorem allows us call conditions (ii) and (iii)
analogs of the Jacobian condition (i) for a subring $R$.

\begin{rem}
\label{r2}
Conditions (ii) and (iii) of Theorem \ref{t2} can be 
expressed in the following way, respetively:

\medskip

\noindent
{\rm (ii)} 
for every irreducible polynomial $w\in k[x_1,\dots,x_m]$
the polynomial $w(f_1,$ $\dots,$ $f_m)$ is square-free,

\medskip

\noindent
{\rm (iii)} 
for every square-free polynomial $w\in k[x_1,\dots,x_m]$
the polynomial $w(f_1,$ $\dots,$ $f_m)$ is square-free.
\end{rem}

\section{Square-factorially closed subrings}
\label{subrings}

We will present some basic observations concerning conditions 
(ii) and (iii) from Theorem \ref{t2} for a subring $R$ 
of an arbitrary unique factorization domain $A$.
Conjecture $\text{JC}(m,n,k)$ motivates us to state 
the following open question.

\medskip

\noindent
{\bf A general question.}
{\em Let $R$ be a subring of a domain $A$ such that
$$\Irr R\subset\Sqf A\quad\mbox{or}\quad\Sqf R\subset\Sqf A.$$
When $R$ is algebraically closed in $A$?}

\medskip

\begin{lm}
\label{l3}
Let $A$ be a unique factorization domain.
Let $R$ be a subring of~$A$ such that $R^{\ast}=A^{\ast}$.
Consider the following conditions:

\medskip

\noindent
{\rm (i)} \
$\Irr R\subset\Irr A$,

\medskip

\noindent
{\rm (ii)} \
$\Sqf R\subset\Sqf A$,

\medskip

\noindent
{\rm (iii)} \
$\Irr R\subset\Sqf A$.

\medskip

Then the following implications hold:
$$\text{\rm (i)}\Rightarrow \text{\rm (ii)}\Rightarrow
\text{\rm (iii)}.$$
\end{lm}

\begin{proof}
$\text{(i)}\Rightarrow \text{(ii)}$

Assume that $\Irr R\subset\Irr A$
and consider an element $a\in\Sqf R$.
Let $a=q_1\ldots q_r$, where $q_1,\dots,q_r\in\Irr R$
are pairwise non-associated in $R$.
Then, by the assumption, $q_1,\dots,q_r\in\Irr A$.
Moreover, since $A^{\ast}=R^{\ast}$,
$q_1,\dots,q_r$ are pairwise non-associated in $A$.

\medskip

\noindent
$\text{(ii)}\Rightarrow \text{(iii)}$

Assume that $\Sqf R\subset\Sqf A$
and consider an element $a\in\Irr R$.
Suppose that $a=b^2c$, where
$b\in R\setminus R^{\ast}$ and $c\in R$.
Then $a=b\cdot (bc)$, where $b,bc\in R\setminus R^{\ast}$,
a contradiction.
Hence, $a\in\Sqf R$.
\end{proof}

Recall that a subring $R$ of a ring $A$ is called
{\em factorially closed} in $A$ if the following implication:
$$xy\in R\setminus\{0\}\Rightarrow x,y\in R$$
holds for every $x,y\in A$.
The ring of constants of any locally nilpotent derivation
of a domain of characteristic zero is factorially closed.
We refer the reader to \cite{DaigleBook} and 
\cite{FreudenburgBook} for more information about
locally nilpotent derivations.

\begin{lm}
\label{l4}
Let $A$ be a unique factorization domain.
Let $R$ be a subring of~$A$ such that $R^{\ast}=A^{\ast}$.
The following conditions are equivalent:

\medskip

\noindent
{\rm (i)} \
$\Irr R\subset\Irr A$,

\medskip

\noindent
{\rm (ii)} \
$R$ is factorially closed in $A$.
\end{lm}

\begin{proof}
$\text{(i)}\Rightarrow \text{(ii)}$

Assume that $\Irr R\subset\Irr A$ and consider elements
$x,y\in A$ such that $xy\in R\setminus\{0\}$.
Let $xy=q_1\ldots q_r$, where $q_1,\dots,q_r\in\Irr R$.
Then $q_1,\dots,q_r\in\Irr A$,
so without loss of generality we may assume that
$x=aq_1\ldots q_s$ and $y=bq_{s+1}\ldots q_r$
for some $a,b\in A^{\ast}$.
Since $A^{\ast}=R^{\ast}$, we infer that $x,y\in R$.

\medskip

\noindent
$\text{(ii)}\Rightarrow \text{(i)}$

Assume that condition (ii) holds and consider
an element $a\in\Irr R$.
Let $a=bc$, where $b,c\in A$.
By the assumption, $b,c\in R$,
so $b\in R^{\ast}$ or $c\in R^{\ast}$.
Hence, $b\in A^{\ast}$ or $c\in A^{\ast}$.
\end{proof}

Note the following easy fact.

\begin{lm}
\label{l5}
Let $A$ be a domain, let $R$ be a subring of $A$.
The following conditions are equivalent:

\medskip

\noindent
{\rm (i)} \
$R_0\cap A=R$,

\medskip

\noindent
{\rm (ii)} \
for every $x\in R$, $y\in A$,
if $xy\in R\setminus\{0\}$, then $y\in R$.
\end{lm}

\begin{tw}
\label{t3}
Let $A$ be a unique factorization domain.
Let $R$ be a subring of $A$ such that
$R^{\ast}=A^{\ast}$ and $R_0\cap A=R$.
The following conditions are equivalent:

\medskip

\noindent
{\rm (i)} \
$\Sqf R\subset\Sqf A$,

\medskip

\noindent
{\rm (ii)} \
for every $x\in A$, $y\in\Sqf A$,
if $x^2y\in R\setminus\{0\}$, then $x,y\in R$.
\end{tw}

\begin{proof}
$\text{(i)}\Rightarrow \text{(ii)}$

Assume that $\Sqf R\subset\Sqf A$ and consider $x,y\in A$
such that $y\in\Sqf A$ and $x^2y\in R\setminus\{0\}$.
If $x\in R$, then $x^2\in R$, and hence $y\in R$
by Lemma~\ref{l5}.

\medskip

Now suppose that $x\not\in R$ and $x$ is a minimal element
(with respect to a~number of irreducible factors in $A$)
with this property.
In this case $x\not\in A^{\ast}$, so $x^2y\not\in\Sqf A$,
and then $x^2y\not\in\Sqf R$.
Hence, $x^2y=z^2t$ for some $z,t\in R$, $z\not\in R^{\ast}$.
We can present $t$ in the form $t=u^2v$,
where $u,v\in A$, $v\in\Sqf A$.
We have $u^2v\in R\setminus\{0\}$,
so $u\in R$ by the minimality of $x$.
We obtain $x^2y=z^2u^2v$, where $y,v\in\Sqf A$,
hence $x=czu$ for some $c\in A^{\ast}$.
By the assumptions, $c\in R$, so $x\in R$, a contradiction.

\medskip

\noindent
$\text{(ii)}\Rightarrow \text{(i)}$

Assume that condition (ii) holds.
Consider an element $r\in\Sqf R$.
Suppose that $r\not\in\Sqf A$, so $r=x^2y$ for some
$x,y\in A$ such that $x\not\in A^{\ast}$ and $y\in\Sqf A$.
Since $x^2y\in R\setminus\{0\}$, we obtain that $x,y\in R$.
We have $x\not\in R^{\ast}$, so $x^2y\not\in\Sqf R$,
a contradiction.
\end{proof}

\begin{df}
Let $A$ be a UFD.
A subring $R$ of $A$ such that the implication
$$x^2y\in R\setminus\{0\}\;\Rightarrow\;x,y\in R$$
holds for every $x\in A$, $y\in\Sqf A$,
will be called {\em square-factorially closed} in~$A$.
\end{df}

Recall that a subring $R$ of a ring $A$ is called
{\em root closed} \ in $A$ if the following implication:
$$x^n\in R\;\Rightarrow\;x\in R$$
holds for every $x\in A$ and $n\geqslant 1$.
The properties of root closed subrings were investigated
in many papers, see \cite{Anderson} and \cite{BrewerCMcC}
for example.

\begin{tw}
\label{t4}
Let $A$ be a unique factorization domain.
Let $R$ be a subring of $A$ such that
$R^{\ast}=A^{\ast}$ and $R_0\cap A=R$.
If $R$ is square-free closed in $A$,
then $R$ is root closed in $A$.
\end{tw}

\begin{proof}
Assume that $R$ is square-free closed in $A$.
Consider an element $x\in A$, $x\neq 0$,
such that $x^n\in R$ for some $n\geqslant 1$.
Let $n=2^r(2l+1)$, where $r,l\geqslant 0$.
Observe first that since $(x^{2l+1})^{2^r}\in R$,
then applying Theorem~\ref{t3} we obtain
$x^{2l+1}\in R$.

\medskip

Now, note that $x$ can be presented in the form
$x=s_0^{2^m}s_1^{2^{m-1}}\ldots s_{m-1}^2s_m$,
where $s_0,\dots,s_m\in \Sqf A$ and $m\geqslant 0$.
We will show by induction on $m$ that
the following implication:
$$x^{2l+1}\in R\;\Rightarrow\;x\in R$$
holds for every $x\in A$, $x\neq 0$.
Let $m>0$ and assume that the above implication holds for $m-1$.
Put $t=s_0^{2^{m-1}}s_1^{2^{m-2}}\ldots s_{m-1}$, so $x=t^2s_m$
and $x^{2l+1}=(t^{2l+1}s_m^l)^2s_m$.
If $x^{2l+1}\in R$,
then $t^{2l+1}s_m^l,s_m\in R$ by Theorem~\ref{t3},
so $t^{2l+1}\in R$ by Lemma~\ref{l5}.
Hence, $t\in R$ by the induction assumption
and, finally, $x\in R$.
\end{proof}

\bigskip

\noindent
{\bf Acknowledgements.}
The authors would like to thank Prof.\ Holger Brenner
for valuable discussions on the matter of this paper.

\end{document}